\documentclass[a4paper,10pt]{amsart}
\usepackage{amssymb,amsthm,amsmath}
\usepackage{ifthen}
\usepackage{graphicx}
\nonstopmode
 \numberwithin{equation}{section}
\newtheorem{thm}{Theorem}[section]
\newtheorem{cor}[thm]{Corollary}
\newtheorem{lem}[thm]{Lemma}
\newtheorem{prop}[thm]{Proposition}
\newtheorem{defn}[thm]{Definition}

\theoremstyle{definition}
\newtheorem{rmk}[thm]{Remark}

{\qed\bigskip}

\newcounter{alphabet}
\newcounter{tmp}

\ifx\undefined\bysame
\newcommand{\bysame}{\leavevmode\hbox to3em{\hrulefill}\,}
\fi

\begin{document}
\baselineskip=21pt
%\hfill{pps04.tex}
\markboth{} {}

\title[Some identities and inequalities for Hilbert-Schmidt frames]
%as sampling density tends to infinity
{Some identities and inequalities for Hilbert-Schmidt frames}
%as sampling density tends to infinity

\author{Anirudha Poria}

\address{Department of Mathematics,
Indian Institute of Technology Guwahati,
Guwahati 781039, India.}
\email{a.poria@iitg.ernet.in}

\begin{abstract} 
In this paper we establish Parseval type identities and surprising new inequalities for Hilbert-Schmidt frames. Our results generalize and improve the remarkable results which have been obtained by Balan et al. and G{\u{a}}vru{\c{t}}a.  
\end{abstract}
%\thanks{}
\date{\today}
\maketitle
\def\BC{{\mathbb C}} \def\BQ{{\mathbb Q}}
\def\BR{{\mathbb R}} \def\BI{{\mathbb I}}
\def\BZ{{\mathbb Z}} \def\BD{{\mathbb D}}
\def\BP{{\mathbb P}} \def\BB{{\mathbb B}}
\def\BS{{\mathbb S}} \def\BH{{\mathbb H}}
\def\BE{{\mathbb E}}  \def\BK{{\mathbb K}}
\def\BN{{\mathbb N}}
\def\g{{\mathcal{G}_j}}
\def\gtr{{\mathcal{G}_j^*}}
\def\gt{{\mathcal{\tilde{G}}_j}}
\def\ga{{\varGamma_j}}
%In above style or below style, you can define
%\newcommand{\cal {B}}{\mathcal{B}}
%\mathcal{}
\vspace{-.5cm}

\section{Introduction}
The concept of a frame in Hilbert spaces has been introduced in 1952 by Duffin and Schaeffer \cite{duf52}, in the context of nonharmonic Fourier series (see \cite{you01}). After the work of Daubechies et al. \cite{dau86} frame theory got considerable attention outside signal processing and began to be more broadly studied (see \cite{cas00, chr13, gro01}). A frame for a Hilbert space is a redundant set of vectors in Hilbert space which provides non-unique representations of vectors in terms of frame elements. The redundancy and flexibility offered by frames has spurred their application in several areas of mathematics, physics, and engineering such as sigma-delta quantization \cite{ben06}, neural networks \cite{can99}, image processing \cite{can05}, system modelling \cite{dud98}, quantum measurements \cite{eld02}, sampling theory \cite{fei94}, wireless communications \cite{str03} and many other well known fields.

Throughout this paper, $\BH$ and $\BK$ are two Hilbert spaces, $\mathcal{L}(\BH)$ the algebra of all bounded linear operators on $\BH$, $I$ the identity operator on $\BH$, and $J$ is a countable index set. Now we recall the definition and some basic properties of frames in Hilbert spaces. 
\begin{defn}
A family $\{f_j: j \in J \}$ in $\BH$ is called a frame for $\BH$, if there exist constants $0 < A \leq B < \infty $ such that for all $f \in \BH$
\begin{equation}
A \Vert f \Vert^2 \leq \sum_{j \in J } \vert \langle f,f_j \rangle \vert^2 \leq  B \Vert f \Vert^2.
\end{equation}
\end{defn} 
The constants $A$ and $B$ are called $frame$ $bounds$. If $A=B$, then this frame is called an $A$-$tight$ $frame$, and if $A=B=1$, then it is called a $Parseval$ $frame$.

If $\{f_j: j \in J \}$ is a frame for $\BH$, then the following three operators are linear and bounded:\\
$Synthesis$ $operator:$ \;\; $T: l^2(J)\rightarrow \BH$,  \;\;\; $T(\{c_j\}_{j \in J})=\sum\limits_{j \in J}c_jf_j$,\\
$Analysis$ $operator:$ \;\; $T^*: \BH \rightarrow l^2(J)$, \;\;\; $T^* f=\{ \langle f,f_j \rangle\}_{j \in J} $,\\
$Frame$ $operator:$ \;\;\;\; $S:\BH \rightarrow \BH,$ \;\;\;\;\;\;\;\; $Sf=TT^*f=\sum\limits_{j \in J}\langle f,f_j \rangle f_j.$\\
Moreover, $T^*$ is the adjoint of $T$, and $S$ is a positive self-adjoint invertible operator in $\BH$. The following reconstruction formula holds for all $f \in \BH$:
\begin{equation}
f=\sum\limits_{j \in J}\langle f,f_j \rangle S^{-1}f_j=\sum\limits_{j \in J}\langle f,S^{-1}f_j \rangle f_j,
\end{equation}
where the family $\{\tilde{f_j}=S^{-1}f_j: j \in J \}$ is also a frame for $\BH$, which is called the $canonical$ $dual$ $frame$ of $\{f_j: j \in J \}$. A frame $\{g_j : j \in J \}$ for $\BH$ is called an $alternate$ $dual$ $frame$ of $\{f_j: j \in J \}$ if for all $f \in \BH$ the following equality holds: 
\begin{equation}
f=\sum\limits_{j \in J}\langle f,g_j \rangle f_j.
\end{equation}
Let $\{f_j: j \in J \}$ be a frame for $\BH$. For every $K \subset J$, we define the operator $S_K$ by 
\begin{equation}
S_Kf=\sum\limits_{j \in K}\langle f,f_j \rangle f_j,
\end{equation}
and also we denote $K^c$ as $J \setminus K.$

We refer to \cite{cas00, chr13, dau92, gro01,  han00, por16, you01} for basic results on frames and \cite{ask01, cas99,  cas04, chr03, gab03, kaf09, por15, sad12, sun06} for generalizations of frames.

In \cite{bal06}, the authors proved a longstanding conjecture of the signal processing community: a signal can be reconstructed without information about the phase. While working on efficient algorithms for signal reconstruction, Balan et al. \cite{bal07} discovered a remarkable new identity for Parseval frames, given in the following form. (We refer to \cite{bal05} for a discussion of the origins of this fundamental identity.)

\begin{thm}\label{th1}
Let $\{ f_j:j \in J \}$ be a Parseval frame for $\BH$, then for every $K\subset J$ and every $f \in \BH$, we have
\[ \sum_{j \in K} \vert \langle f,f_j \rangle \vert^2- \bigg\Vert \sum_{j \in K} \langle f,f_j \rangle f_j \bigg\Vert^2=\sum_{j \in K^c} \vert \langle f,f_j \rangle \vert^2- \bigg\Vert \sum_{j \in K^c} \langle f,f_j \rangle f_j \bigg\Vert^2. \]
\end{thm}
\begin{thm}\label{th001}
If $\{ f_j:j \in J \}$ be a Parseval frame for $\BH$, then for every $K\subset J$ and every $f \in \BH$, we have
\[ \sum_{j \in K} \vert \langle f,f_j \rangle \vert^2 + \bigg\Vert \sum_{j \in K^c} \langle f,f_j \rangle f_j \bigg\Vert^2 \geq \frac{3}{4} \Vert f \Vert^2. \]
\end{thm}
Because of the importance of Parseval frames in applications, particularly to signal processing the authors in \cite{bal07} mainly focused on Parseval frames and proved several interesting variants of Theorem \ref{th1}. In fact, the identity that appears in Theorem \ref{th1} was obtained in \cite{bal07} as a particular case of the following result for general frames.
\begin{thm}\label{th01}
Let $\{ f_j:j \in J \}$ be a frame for $\BH$ with canonical dual frame $\{ \tilde{f_j}: j \in J \}$. Then for every $K\subset J$ and every $f \in \BH$, we have
\[ \sum_{j \in K} \vert \langle f,f_j \rangle \vert^2- \sum_{j \in J} \vert \langle S_K f, \tilde{f_j}  \rangle \vert^2 = \sum_{j \in K^c} \vert \langle f,f_j \rangle \vert^2- \sum_{j \in J} \vert \langle S_{K^c} f, \tilde{f_j}  \rangle \vert^2 .\]
\end{thm}
The following results, which were obtained in \cite{gua06}, generalize Theorems \ref{th1} and \ref{th001} to canonical and alternate dual frames:
\begin{thm}\label{th05}
Let $\{ f_j:j \in J \}$ be a frame for $\BH$ with canonical dual frame $\{ \tilde{f_j}: j \in J \}$. Then for every $K\subset J$ and every $f \in \BH$, we have
\[ \sum_{j \in K} \vert \langle f,f_j \rangle \vert^2+\sum_{j \in J} \vert \langle S_{K^c} f, \tilde{f_j}  \rangle \vert^2= \sum_{j \in K^c} \vert \langle f,f_j \rangle \vert^2 +\sum_{j \in J} \vert \langle S_K f, \tilde{f_j}  \rangle \vert^2 \geq \frac{3}{4} \sum_{j \in J}|\langle f,f_j \rangle|^2  .\]
\end{thm}
\begin{thm}\label{th2}
Let $\{ f_j:j \in J \}$ be a frame for $\BH$ and $\{ g_j:j \in J \}$ be an alternate dual frame of  $\{ f_j:j \in J \}$, then for every $K\subset J$ and every $f \in \BH$, we have
\begin{eqnarray*}
 Re \bigg( \sum_{j \in K} \langle f,g_j \rangle \overline{\langle f, f_j\rangle} \bigg)+ \bigg\Vert \sum_{j \in K^c} \langle f,g_j \rangle f_j \bigg\Vert^2  &=& Re \bigg( \sum_{j \in K^c} \langle f,g_j \rangle \overline{\langle f, f_j\rangle} \bigg)+ \bigg\Vert \sum_{j \in K} \langle f,g_j \rangle f_j \bigg\Vert^2 \\ 
 & \geq & \frac{3}{4} \Vert f \Vert^2. 
\end{eqnarray*} 
\end{thm}
Motivated by these interesting results, the authors in \cite{zhu10} generalized Theorem \ref{th2} to a form that does not involve the real parts of the complex numbers, which is given below.
\begin{thm}\label{th3}
Let $\{ f_j:j \in J \}$ be a frame for $\BH$ and $\{ g_j:j \in J \}$ be an alternate dual frame of  $\{ f_j:j \in J \}$. Then for every $K\subset J$ and every $f \in \BH$, we have
\[  \bigg( \sum_{j \in K} \langle f,g_j \rangle \overline{\langle f, f_j\rangle} \bigg)- \bigg\Vert \sum_{j \in K} \langle f,g_j \rangle f_j \bigg\Vert^2 = \overline{\bigg( \sum_{j \in K^c} \langle f,g_j \rangle \overline{\langle f, f_j\rangle} \bigg)}- \bigg\Vert \sum_{j \in K^c} \langle f,g_j \rangle f_j \bigg\Vert^2. \]
\end{thm}
Moreover, the authors in \cite{li08, xia08} have extended Theorem \ref{th01} for g-frames and canonical dual g-frames in Hilbert spaces. Also, the authors in \cite{yan09} have established an equality and an inequality for the alternate dual g-frame. Further, in \cite{li12}, the authors generalized the equality and inequality for g-frame to a g-Bessel sequence in Hilbert spaces. 

In this paper, we generalize the above mentioned results for Hilbert-Schmidt frames. Also, we generalize the above inequalities to a more general form which involve a scalar $ \lambda \in [0,1]$. As a particular case, for $\lambda = 1/2$, the above inequalities can be obtained. Since g-frames can be considered as a class of Hilbert-Schmidt frames, the previous equality and inequalities on g-frames can be obtained as a special case of the results we establish on Hilbert-Schmidt frames. The paper is organized in the following way: In section 2, we provide some necessary background on Hilbert-Schmidt frame and then, in section 3, we state and prove the main results. 

\section{Hilbert-Schmidt frames}
Let us denote $\{ \BK_j: j \in J \} \subset \BK$ as a sequence of Hilbert spaces and $\mathcal{L}(\BH,\BK_j)$ the collection of all bounded linear operators from $\BH$ to $\BK_j.$ The notion of a frame extended to g-frame by W. Sun \cite{sun06}. First we recall the definition of a g-frame. 
\begin{defn} \cite{sun06}
A family $\{ \Lambda_j \in \mathcal{L}(\BH,\BK_j):j \in J \}$ is called a generalized frame, or simply a g-frame, for $\BH$ with respect to $\{ \BK_j: j \in J \}$ if there are two constants $A, B>0$ such that for all $f \in \BH$
\begin{equation}
A \Vert f \Vert^2 \leq \sum_{j \in J }  \Vert \Lambda_j(f)  \Vert^2 \leq  B \Vert f \Vert^2.
\end{equation} 
\end{defn}
For more details on g-frames see \cite{sun06}.
Let $\mathcal{L}(\BH)$ denote the $C^*$-algebra of all bounded linear operators on a complex separable Hilbert space $\BH$. For a compact operator $T \in \mathcal{L}(\BH)$, the eigenvalues of the positive operator $\vert T \vert = (T^*T)^{1/2}$ are called the singular values of $T$ and denoted by $s_j(T)$. We arrange the singular values $s_j(T)$ in a decreasing order and repeated according to multiplicity, that is, $s_1(T)\geq s_2(T) \geq ... \geq 0$. For $1 \leq p < \infty$, the von Neumann-Schatten p-class $C_p$ is defined to be the set of all compact operators $T$ for which 
\begin{equation}
\Vert T \Vert_p=(\tau \vert T \vert^p)^{\frac{1}{p}}=\bigg( \sum_{j=1}^{\infty} s_j^p(T) \bigg)^{\frac{1}{p}} < \infty,
\end{equation} 
where $\tau$ is the usual trace functional defined as $\tau(T)=\sum\limits_{e \in E}\langle T(e),e \rangle$, and $E$ is any orthonormal basis of $\BH$. For $p=\infty$, let $C_{\infty}$ denote the class of all compact operators with $\Vert T \Vert_{\infty}=s_1(T)< \infty$. For more information about a von Neumann-Schatten ${p}$-class see \cite{rin71, sim79}. We recall that $C_2$ is a Banach space with respect to $\Vert . \Vert_2$, and also it is a Hilbert space with the inner product defined by $\big[ T,S \big]_{\tau}=\tau(S^*T)$.
\begin{defn} \cite{sad12}
A family $\{\g : j \in J\}$ of bounded linear operators from $\BH$ to $C_2 \subseteq \mathcal{L}(\BK)$ is said to be a Hilbert-Schmidt frame, or simply a HS-frame for $\BH$ with respect to $\BK$, if there exist constants $A, B >0$ such that for all $f \in \BH$
\begin{equation}
A \Vert f \Vert^2 \leq \sum_{j \in J }  \Vert \g(f)  \Vert_2^2 \leq  B \Vert f \Vert^2.
\end{equation}
\end{defn}
If $A=B=1$, then $\{\g : j \in J\}$ is called the Parseval HS-frame for $\BH$ with respect to $\BK$. The HS-frame operator $S:\BH \rightarrow \BH$ is defined by $Sf=\sum\limits_{j \in J} \gtr\g(f)$, $f \in \BH$, where $\gtr$ is the adjoint operator of $\g$. If $\{\g : j \in J\}$ is a HS-frame, then S is a bounded, invertible, self-adjoint and positive operator. Also, the following reconstruction formula holds for all $f \in \BH$ 
\begin{equation}
f=SS^{-1}f=S^{-1}Sf=\sum_{j \in J} \gtr \g S^{-1}f=\sum\limits_{j \in J} S^{-1}\gtr \g f.
\end{equation}
We call $\{\gt = \g S^{-1}: j \in J \}$ the canonical dual HS-frame of $\{\g : j \in J\}$. A HS-frame $\{\ga : j \in J\}$ is called an alternate dual HS-frame of $\{\g : j \in J \}$ if for all $f \in \BH$ the following identity holds: 
\begin{equation}\label{eq01}
f=\sum\limits_{j \in J}\gtr \ga f=\sum\limits_{j \in J} \varGamma_j^{*}\g f .
\end{equation}

Let $\{\g : j \in J\}$ be a HS-frame. For every $K\subset J$, define the bounded linear operators $S_K, S_{K^c}: \BH \rightarrow \BH$ by \[ S_Kf=\sum_{j \in K} \gtr \g(f), \;\;\; S_{K^c}f=\sum_{j \in K^c} \gtr \g(f).\]
It is easy to check that $S_K$ and $S_{K^c}$ are self-adjoint. 
\begin{lem}\label{lem00}
\cite{sad12} Let $\{\Lambda_j :j \in J\}$ be a g-frame for $\BH$ with respect to $\{ \BK_j: j \in J \}$. Then $\{\Lambda_j :j \in J\}$ is a HS-frame for $\BH$ with respect to $\BK=\bigoplus\limits_{j \in J}\BK_j.$
\end{lem}
In \cite{sun06}, W. Sun has shown that bounded quasi-projectors \cite{for04}, frames of subspaces \cite{cas04}, pseudo-frames \cite{li04}, oblique frames \cite{chr04}, outer frames \cite{ald04}, and time-frequency localization operators \cite{dor06} are special classes of g-frames. Hence, Lemma \ref{lem00} implies that each of these classes is also a class of HS-frames.

\section{The main results and their proofs}
We first state a simple result on operators, which can be found in \cite{zhu10}. 
\begin{lem}\label{lem1}
If $P,Q \in \mathcal{L}(\BH)$ satisfying $P+Q=I$, then $P-P^{*}P=Q^*-Q^*Q.$
\end{lem}
\begin{proof}
We compute $P-P^{*}P=(I-P^{*})P=Q^*(I-Q)=Q^*-Q^*Q.$
\end{proof}
Now we state and prove a Parseval HS-frame identity.
\begin{thm}
Let $\{ \g : j \in J \}$ be a Parseval HS-frame for $\BH$ with respect to $\BK$. Then for all $K \subset J$ and all $f \in \BH$, we have
\[ \sum_{j \in K}\Vert \g(f) \Vert^2 - \bigg\Vert \sum_{j \in K} \gtr \g f \bigg\Vert^2 = \sum_{j \in K^c}\Vert \g(f) \Vert^2 - \bigg\Vert \sum_{j \in K^c} \gtr \g f \bigg\Vert^2. \]
\end{thm}
\begin{proof}
Since $\{ \g : j \in J \}$ is a Parseval HS-frame, the corresponding frame operator $S=I$, and hence $S_K+S_{K^c}=I$. Note that $S_{K^c}$ is a self-adjoint operator, and therefore $S_{K^c}^*=S_{K^c}$. Applying Lemma \ref{lem1} to the operators $S_K$ and $S_{K^c}$, we obtain that for all $f \in \BH$
\begin{eqnarray*}
&& \langle S_K f,f \rangle - \langle S_K^{*}S_K f,f \rangle =\langle S_{K^c}^* f, f \rangle -\langle S_{K^c}^* S_{K^c}f,f \rangle \\
\Rightarrow && \langle S_K f,f \rangle  - \Vert S_K f \Vert^2 = \langle S_{K^c} f,f \rangle  - \Vert S_{K^c}f \Vert^2  \\
\Rightarrow && \sum_{j \in K}\Vert \g(f) \Vert^2 - \bigg\Vert \sum_{j \in K} \gtr \g f \bigg\Vert^2 = \sum_{j \in K^c}\Vert \g(f) \Vert^2 - \bigg\Vert \sum_{j \in K^c} \gtr \g f \bigg\Vert^2.
\end{eqnarray*}
Hence we have the desired result.
\end{proof}
The next inequality for a Parseval HS-frame, appearing in Corollary \ref{th002}, is a simple consequence of Theorem \ref{th001}.
\begin{cor}\label{th002}
Let $\{ \g : j \in J \}$ be a Parseval HS-frame for $\BH$ with respect to $\BK$. Then for all $K \subset J$ and all $f \in \BH$, we have \[ \sum_{j \in K}\Vert \g(f) \Vert^2 + \bigg\Vert \sum_{j \in K^c} \gtr \g f \bigg\Vert^2 \geq \frac{3}{4} \Vert f \Vert^2. \]
\end{cor}
\begin{proof}
Since $\{ \g : j \in J \}$ is a Parseval HS-frame, $S_K+S_{K^c}=I$. A simple computation shows that
\[S^2_K+S^2_{K^c}=S^2_K+\left(I-S_{K}\right)^2=2S^2_K -2 S_K+I=2 \bigg( S_K- \frac{1}{2}I \bigg)^2+\frac{1}{2} I, \]
and so \[ S^2_K+S^2_{K^c} \geq \frac{1}{2}I. \]
Since $S_K+S_{K^c}=I$, it follows that $S_K+S^2_{K^c}+S_{K^c}+S^2_{K} \geq \frac{3}{2} I.$ Notice that operators $S_K$ and $S_{K^c}$ are self-adjoint and therefore $S_{K}^*=S_{K}$, $S_{K^c}^*=S_{K^c}$.  Applying Lemma \ref{lem1} to the operators $P=S_K$ and $Q=S_{K^c}$, we obtain
\[S_K-S^2_K=S_{K^c}-S^2_{K^c} \Rightarrow S_K+S^2_{K^c}=S_{K^c}+S^2_K. \]
Thus \[ 2(S_K+S^2_{K^c})=S_K+S^2_{K^c}+S_{K^c}+S^2_{K} \geq \frac{3}{2} I .\] 
Therefore for all $f \in \BH$ we have
\[\sum_{j \in K}\Vert \g(f) \Vert^2 + \bigg\Vert \sum_{j \in K^c} \gtr \g f \bigg\Vert^2= \langle S_Kf,f \rangle+ \langle S_{K^c}f,  S_{K^c}f \rangle=\langle (S_K+S^2_{K^c})f,f \rangle  \geq \frac{3}{4}\Vert f \Vert^2.\]
This completes the proof.
\end{proof}
Now we generalize Theorems \ref{th1} and \ref{th2} to dual HS-frames. We first establish the following result.
\begin{prop}\label{pro1}
Let $P,Q \in \mathcal{L}(\BH)$ be two self-adjoint operators such that $P+Q=I$. Then for any $\lambda \in [0,1]$ and all $f \in \BH$ we have
\[ \Vert Pf \Vert^2+2 \lambda \langle Qf,f \rangle= \Vert Qf \Vert^2+2(1-\lambda)\langle Pf,f \rangle+ (2\lambda -1)\Vert f \Vert^2 \geq (1-(\lambda-1)^2)\Vert f \Vert^2. \]
\end{prop}
\begin{proof}
We have
\begin{eqnarray*}
\Vert Pf \Vert^2+2 \lambda \langle Qf,f \rangle=\langle P^2f,f \rangle +2 \lambda \langle(I-P)f,f \rangle = \langle( P^2-2 \lambda P+2 \lambda I )f,f \rangle ,
\end{eqnarray*}
and
\begin{eqnarray*}
&& \Vert Qf \Vert^2 +2(1-\lambda)\langle Pf,f \rangle+ (2\lambda -1) \Vert f \Vert^2 \\
 &=& \langle (I-P)^2f,f \rangle +2(1-\lambda)\langle Pf,f \rangle + (2\lambda -1)\langle f,f \rangle \\
 &=& \langle( P^2-2 \lambda P+2 \lambda I )f,f \rangle \\
 &=& \langle( (P-\lambda I)^2 -\lambda^2 I +2 \lambda I )f,f \rangle \\
 &=& \langle( (P-\lambda I)^2 +(1-(\lambda-1)^2) I)f,f \rangle
  \geq  (1-(\lambda-1)^2)\Vert f \Vert^2.
\end{eqnarray*}
This proves the desired result.
\end{proof}
\begin{thm}\label{th02}
Let $\{ \g : j \in J \}$ be a HS-frame for $\BH$ with respect to $\BK$ and $\{ \gt: j \in J \}$ be the canonical dual HS-frame of $\{ \g : j \in J \}$. Then for any $\lambda \in [0,1]$, for all $K \subset J$ and all $f \in \BH$, we have
\begin{eqnarray*}
\sum_{j \in J}\Vert \gt S_{K}f \Vert^2 +\sum_{j \in K^c} \Vert \g(f)\Vert^2  &=& \sum_{j \in J}\Vert \gt S_{K^c}f \Vert^2 +\sum_{j \in K} \Vert \g(f)\Vert^2   \\
&\geq & (2 \lambda- \lambda^2) \sum_{j \in K} \Vert \g(f)\Vert^2 +(1 - \lambda^2) \sum_{j \in K^c} \Vert \g(f)\Vert^2.
\end{eqnarray*} 
\end{thm}
\begin{proof}
Let $S$ be the frame operator for $\{ \g : j \in J \}$. Since $S_k+S_{K^c}=S$, it follows that \[S^{-1/2}S_KS^{-1/2}+S^{-1/2}S_{K^c}S^{-1/2}=I. \] 
Considering $P=S^{-1/2}S_KS^{-1/2}$, $Q=S^{-1/2}S_{K^c}S^{-1/2}$, and $S^{1/2}f$ instead of $f$ in Proposition \ref{pro1}, we obtain
\begin{eqnarray}\label{eq001}
\nonumber
&& \Vert S^{-1/2}S_Kf \Vert^2+2 \lambda \langle S^{-1/2}S_{K^c}f,S^{1/2}f \rangle \\ &=& \Vert S^{-1/2}S_{K^c}f \Vert^2+2(1-\lambda)\langle S^{-1/2}S_K f,S^{1/2}f \rangle+ (2\lambda -1)\Vert S^{1/2}f \Vert^2 \nonumber \\ \nonumber
& \geq & (1-(\lambda-1)^2)\Vert S^{1/2}f \Vert^2 \nonumber \\ \nonumber
\Rightarrow && \langle S^{-1}S_Kf, S_Kf \rangle +2 \lambda \langle S_{K^c}f,f \rangle \\ \nonumber &=& \langle S^{-1}S_{K^c}f, S_{K^c}f \rangle+2(1-\lambda)\langle S_K f,f \rangle+ (2\lambda -1)\langle Sf,f \rangle \\ \nonumber
& \geq & (2 \lambda-\lambda^2)\langle Sf,f \rangle \\ \nonumber
\Rightarrow && \langle S^{-1}S_Kf, S_Kf \rangle \\ \nonumber &=& \langle S^{-1}S_{K^c}f, S_{K^c}f \rangle +2\langle S_K f,f \rangle -2\lambda \langle (S_K+S_{K^c}) f,f \rangle+ (2\lambda -1)\langle Sf,f \rangle \\ \nonumber
& \geq & (2 \lambda-\lambda^2) \langle Sf,f \rangle - 2 \lambda \langle S_{K^c}f,f \rangle \\ \nonumber
\Rightarrow && \langle S^{-1}S_Kf, S_Kf \rangle  = \langle S^{-1}S_{K^c}f, S_{K^c}f \rangle +2\langle S_K f,f \rangle -\langle Sf,f \rangle \\ \nonumber
& \geq & 2 \lambda \langle S_Kf,f \rangle - \lambda^2 \langle Sf,f \rangle   \\ \nonumber
\Rightarrow && \langle S^{-1}S_Kf, S_Kf \rangle +\langle S_{K^c}f,f \rangle  = \langle S^{-1}S_{K^c}f, S_{K^c}f \rangle +\langle S_K f,f \rangle \\ 
& \geq & (2 \lambda- \lambda^2) \langle S_Kf,f \rangle+(1 - \lambda^2) \langle S_{K^c}f,f \rangle .
\end{eqnarray} 
We have 
\begin{eqnarray}\label{eq1}
\langle S^{-1}S_Kf, S_Kf \rangle =\langle SS^{-1}S_Kf, S^{-1}S_Kf \rangle & =& \langle \sum_{j \in J}\gtr \g S^{-1}S_Kf, S^{-1}S_Kf \rangle \nonumber \\
 & = & \sum_{j \in J} \left[ \g S^{-1}S_Kf, \g S^{-1}S_Kf \right]_{\tau} \nonumber \\
 & = & \sum_{j \in J} \left[ \gt S_Kf, \gt S_Kf \right]_{\tau} \nonumber \\
  & = & \sum_{j \in J} \Vert \gt S_Kf \Vert^2 .
\end{eqnarray} 
Similarly 
\begin{eqnarray} \label{eq2}
\langle S^{-1}S_{K^c}f, S_{K^c}f \rangle= \sum_{j \in J} \Vert \gt S_{K^c}f \Vert^2 .
\end{eqnarray}
\begin{eqnarray}\label{eq3}
\langle S_{K^c}f,f \rangle = \sum_{j \in K^c} \Vert \g(f)\Vert^2 .
\end{eqnarray}
\begin{eqnarray}\label{eq4}
\langle S_{K}f,f \rangle = \sum_{j \in K} \Vert \g(f)\Vert^2 .
\end{eqnarray}
Using equations \eqref{eq1}-\eqref{eq4} in the inequality (\ref{eq001}), we obtain
\begin{eqnarray*}
\sum_{j \in J}\Vert \gt S_{K}f \Vert^2 +\sum_{j \in K^c} \Vert \g(f)\Vert^2  &=& \sum_{j \in J}\Vert \gt S_{K^c}f \Vert^2 +\sum_{j \in K} \Vert \g(f)\Vert^2   \\
&\geq & (2 \lambda- \lambda^2) \sum_{j \in K} \Vert \g(f)\Vert^2 +(1 - \lambda^2) \sum_{j \in K^c} \Vert \g(f)\Vert^2 .
\end{eqnarray*}
This completes the proof.
\end{proof}
\begin{prop}\label{prop2}
If $P,Q \in \mathcal{L}(\BH)$ satisfy $P+Q=I$, then for any $\lambda \in [0,1]$ and all $f \in \BH$ we have
\[ P^{*}P+\lambda(Q^{*}+Q)=Q^{*}Q+(1-\lambda)(P^{*}+P)+(2\lambda -1)I \geq (1-(\lambda-1)^2)I. \]
\end{prop}
\begin{proof}
We have
\begin{eqnarray*}
P^{*}P+\lambda(Q^{*}+Q)=P^{*}P+\lambda(I-P^{*}+I-P)=P^{*}P-\lambda(P^{*}+P)+2\lambda I ,
\end{eqnarray*}
and 
\begin{eqnarray*}
 Q^{*}Q+(1-\lambda)(P^{*}+P)+(2\lambda -1)I &=&(I-P^{*})(I-P)+(1-\lambda)(P^{*}+P)+(2\lambda -1)I \\
&=& P^{*}P-\lambda(P^{*}+P)+2 \lambda I \\
&=& (P-\lambda I)^{*}(P-\lambda I)+(1-(\lambda-1)^2)I \\
& \geq & (1-(\lambda-1)^2)I .
\end{eqnarray*}
Hence the result follows.
\end{proof}
\begin{thm}\label{th03}
Let $\{ \g : j \in J \}$ be a HS-frame for $\BH$ with respect to $\BK$ and $\{ \ga: j \in J \}$ be an alternate dual HS-frame of $\{ \g : j \in J \}$. Then for any $\lambda \in [0,1]$, for all $K \subset J$ and all $f \in \BH$, we have
\begin{eqnarray*}
&& Re \bigg\{ \sum_{j \in K^c} \big[ \ga (f),\g (f) \big]_{\tau} \bigg\}+ \bigg\Vert \sum_{j \in K} \gtr \ga (f) \bigg\Vert^2 \\
&=& Re \bigg\{ \sum_{j \in K} \big[ \ga (f),\g (f) \big]_{\tau} \bigg\}+ \bigg\Vert \sum_{j \in K^c} \gtr \ga (f) \bigg\Vert^2 \\
& \geq & (2 \lambda - \lambda^2) Re \bigg\{ \sum_{j \in K} \big[ \ga (f),\g (f) \big]_{\tau} \bigg\}+(1-\lambda^2) Re \bigg\{ \sum_{j \in K^c} \big[ \ga (f),\g (f) \big]_{\tau} \bigg\}.
\end{eqnarray*}
\end{thm}
\begin{proof}
For $K \subset J$ and $f \in \BH$, define the operator $F_{k}$ by $F_{k}f=\sum\limits_{j \in K}\gtr \ga f.$ Then the series converges unconditionally and $F_K \in \mathcal{L}(\BH)$. By (\ref{eq01}), we have $F_K+F_{K^c}=I$. By Proposition \ref{prop2}, we get
\begin{eqnarray*}
  && (1-(\lambda-1)^2)\Vert f \Vert^2 \leq \langle F_K^{*}F_K f,f \rangle +\lambda \langle(F_{K^c}^{*}+F_{K^c})f,f \rangle \\ &=& \langle F_{K^c}^{*}F_{K^c}f,f \rangle + (1-\lambda) \langle (F_K^{*}+F_K)f,f \rangle +(2\lambda -1)\Vert  f \Vert^2 \\
\Rightarrow && (2 \lambda - \lambda^2) Re(\langle If,f \rangle) \leq \Vert F_K f \Vert^2 +\lambda ( \overline{\langle F_{K^c}f,f \rangle } +\langle F_{K^c}f,f \rangle ) \\
&=& \Vert F_{K^c} f \Vert^2 +(1 - \lambda) ( \overline{\langle F_K f,f \rangle } +\langle F_K f,f \rangle )+(2\lambda -1)\Vert  f \Vert^2 \\
\Rightarrow && (2 \lambda - \lambda^2) Re(\langle (F_K+F_{K^c})f,f \rangle) \leq \Vert F_K f \Vert^2 +2 \lambda Re (\langle F_{K^c}f,f \rangle ) \\
&=& \Vert F_{K^c} f \Vert^2 + 2 (1- \lambda) Re (\langle F_K f,f \rangle) +(2\lambda -1)\Vert  f \Vert^2 \\
\Rightarrow && (2 \lambda - \lambda^2) Re(\langle F_Kf,f \rangle) -\lambda^2 Re(\langle F_{K^c}f,f \rangle)  \leq
\Vert F_K f \Vert^2 \\
&=& \Vert F_{K^c} f \Vert^2 +2 Re (\langle F_K f,f \rangle) - Re(\langle If,f \rangle) \\
\Rightarrow && (2 \lambda - \lambda^2) Re(\langle F_Kf,f \rangle) -\lambda^2 Re(\langle F_{K^c}f,f \rangle)  \leq \Vert F_K f \Vert^2 \\
&=& \Vert F_{K^c} f \Vert^2 +2 Re (\langle F_K f,f \rangle) - Re(\langle (F_K+F_{K^c})f,f \rangle) \\
\Rightarrow && (2 \lambda - \lambda^2) Re(\langle F_Kf,f \rangle)+(1 -\lambda^2) Re(\langle F_{K^c}f,f \rangle)  \leq \Vert F_K f \Vert^2 + Re (\langle F_{K^c }f,f \rangle) \\
&=& \Vert F_{K^c} f \Vert^2 + Re (\langle F_K f,f \rangle) .
\end{eqnarray*}
We have
\begin{eqnarray*}
\langle F_Kf,f \rangle=\langle \sum\limits_{j \in K}\gtr \ga f,f \rangle=\sum\limits_{j \in K} \big[ \ga f, \g f \big]_{\tau} .\\
\langle F_{K^c}f,f \rangle=\sum\limits_{j \in K^c} \big[ \ga f, \g f \big]_{\tau} .
\end{eqnarray*}
So finally 
\begin{eqnarray*}
&& Re \bigg\{ \sum_{j \in K^c} \big[ \ga (f),\g (f) \big]_{\tau} \bigg\}+ \bigg\Vert \sum_{j \in K} \gtr \ga (f) \bigg\Vert^2 \\
&=& Re \bigg\{ \sum_{j \in K} \big[ \ga (f),\g (f) \big]_{\tau} \bigg\}+ \bigg\Vert \sum_{j \in K^c} \gtr \ga (f) \bigg\Vert^2 \\
& \geq & (2 \lambda - \lambda^2) Re \bigg\{ \sum_{j \in K} \big[ \ga (f),\g (f) \big]_{\tau} \bigg\}+(1-\lambda^2) Re \bigg\{ \sum_{j \in K^c} \big[ \ga (f),\g (f) \big]_{\tau} \bigg\} .
\end{eqnarray*}
This completes the proof.
\end{proof}
\begin{rmk}
If we consider $\lambda =1/2$ in Theorem \ref{th02} and Theorem \ref{th03}, then we obtain the similar inequalities as in Theorem \ref{th05} and Theorem \ref{th2} respectively, with scalar $3/4$.
\end{rmk}
Next we give a simplified presentation of Theorem \ref{th3} for HS-frames, which generalizes Theorem \ref{th03} to a more general form that does not involve the real parts of the complex numbers. We first establish the following result.
\begin{lem}\label{lem2}
If $P,Q \in \mathcal{L}(\BH)$ such that $P+Q=I$, then $P+Q^*Q=Q^*+P^{*}P.$
\end{lem}
\begin{proof}
By simple computation, we obtain \[P+Q^*Q=P+(I-P^*)(I-P)=(I-P^*)+P^*P=Q^*+P^{*}P,\]
which is as required.
\end{proof}
\begin{thm}\label{th5}
Let $\{ \g : j \in J \}$ be a HS-frame for $\BH$ with respect to $\BK$ and $\{ \ga: j \in J \}$ be an alternate dual HS-frame of $\{ \g : j \in J \}$, then for every $K \subset J$ and every $f \in \BH$, we have
\[ \bigg( \sum_{j \in K^c} \big[ \ga (f),\g (f) \big]_{\tau} \bigg)+ \bigg\Vert \sum_{j \in K} \gtr \ga (f) \bigg\Vert^2 
= \overline{\bigg( \sum_{j \in K} \big[ \ga (f),\g (f) \big]_{\tau} \bigg)}+ \bigg\Vert \sum_{j \in K^c} \gtr \ga (f) \bigg\Vert^2 \]
\end{thm}
\begin{proof}
For $K \subset J$ and $f \in \BH$, we define the operator $F_{k}$ as in Theorem \ref{th03}. Therefore, we have $F_K+F_{K^c}=I$. By Lemma \ref{lem2}, we have
\begin{eqnarray*}
\bigg( \sum_{j \in K^c} \big[ \ga (f),\g (f) \big]_{\tau} \bigg)+ \bigg\Vert \sum_{j \in K} \gtr \ga (f) \bigg\Vert^2 &=& \langle F_{K^c}f,f \rangle+\langle F_K^* F_K f,f \rangle \\
&=& \langle F_K^*f,f \rangle+\langle F_{K^c}^* F_{K^c} f,f \rangle \\
&=& \overline{ \langle F_K f,f \rangle}+\Vert F_{K^c} f \Vert^2 \\
&=& \overline{\bigg( \sum_{j \in K} \big[ \ga (f),\g (f) \big]_{\tau} \bigg)}+ \bigg\Vert \sum_{j \in K^c} \gtr \ga (f) \bigg\Vert^2.
\end{eqnarray*}
Hence the relation stated in the theorem holds.
\end{proof}
\begin{thm}\label{th6}
Let $\{ \g : j \in J \}$ be a HS-frame for $\BH$ with respect to $\BK$ and $\{ \ga: j \in J \}$ be an alternate dual HS-frame of $\{ \g : j \in J \}$. Then for every bounded sequence $\{w_j :j \in J\}$ and every $f \in \BH$, we have
\begin{eqnarray*}
 && \bigg( \sum_{j \in J} w_j \big[ \ga (f),\g (f) \big]_{\tau} \bigg)+ \bigg\Vert \sum_{j \in J} (1-w_j) \gtr \ga (f) \bigg\Vert^2  \\
&=& \overline{\bigg( \sum_{j \in J}(1-w_j) \big[ \ga (f),\g (f) \big]_{\tau} \bigg)}+ \bigg\Vert \sum_{j \in J} w_j \gtr \ga (f) \bigg\Vert^2.
\end{eqnarray*}
\end{thm}
\begin{proof}
We define the operators $Ff=\sum\limits_{j \in J}w_j \gtr \ga f$ and $Gf=\sum\limits_{j \in J}(1-w_j) \gtr \ga f$. Note that both series converge unconditionally. Also we have $F, G \in \mathcal{L}(\BH)$ and $F+G=I.$ By Lemma \ref{lem2}, we have
\begin{eqnarray*}
&& \bigg( \sum_{j \in J} w_j \big[ \ga (f),\g (f) \big]_{\tau} \bigg)+ \bigg\Vert \sum_{j \in J} (1-w_j) \gtr \ga (f) \bigg\Vert^2 \\ &=& \langle Ff,f \rangle+\langle G^* G f,f \rangle \\
&=& \langle G^*f,f \rangle+\langle F^* F f,f \rangle \\
&=& \overline{ \langle G f,f \rangle}+\Vert Ff \Vert^2 \\
&=& \overline{\bigg( \sum_{j \in J}(1-w_j) \big[ \ga (f),\g (f) \big]_{\tau} \bigg)}+ \bigg\Vert \sum_{j \in J} w_j \gtr \ga (f) \bigg\Vert^2.
\end{eqnarray*}
Hence the relation holds.
\end{proof}
Observe that if we consider $K \subset J$ and
\[ w_j =
\left\{
	\begin{array}{ll}
		0  & \mbox{if } j \in K \\
		1 & \mbox{if } j \in K^c,
	\end{array}
\right. \] then Theorem \ref{th5} follows from Theorem \ref{th6}.
 
\section*{Acknowledgments}
The author is deeply indebted to Prof. Radu Balan for several valuable suggestions concerning Theorem \ref{th02} and Theorem \ref{th6}. The author also wishes to thank Dr. J. Swain for several fruitful discussions. The author is grateful to the Ministry of Human Resource Development, India for providing the research fellowship and Indian Institute of Technology Guwahati, India for the support provided during the period of this work. Further, the author thanks the anonymous referee for valuable suggestions which helped to improve the paper.

\bibliographystyle{plain}

\begin{thebibliography}{10}

\bibitem{ald04}
A.~Aldroubi, C.~Cabrelli, and U.~Molter. 
\newblock Wavelets on irregular grids with arbitrary dilation matrices and frame atomics for $L^2(\BR^d)$.
\newblock {\em Appl. Comput. Harmon. Anal.}, 17(2):119-–140, 2004.

\bibitem{ask01}
A.~Askari-Hemmat, M.~Dehghan, and M.~Radjabalipour.
\newblock Generalized frames and their redundancy.
\newblock {\em Proc. Amer. Math. Soc.}, 129(4):1143--1147, 2001.

\bibitem{bal06}
R.~Balan, P.G. Casazza, and D.~Edidin.
\newblock On signal reconstruction without phase.
\newblock {\em Appl. Comput. Harmon. Anal.}, 20(3):345--356, 2006.

\bibitem{bal05}
R.~Balan, P.G. Casazza, D.~Edidin, and G.~Kutyniok.
\newblock Decompositions of frames and a new frame identity.
\newblock In {\em Wavelets XI (San Diego, CA, 2005), 379--388, SPIE Proc. 5914,
  SPIE, Bellingham, WA}, 2005.

\bibitem{bal07}
R.~Balan, P.G. Casazza, D.~Edidin, and G.~Kutyniok.
\newblock A new identity for {P}arseval frames.
\newblock {\em Proc. Amer. Math. Soc.}, 135(4):1007--1015, 2007.

\bibitem{ben06}
J.J. Benedetto, A.M. Powell, and {\"O}.~Y{\i}lmaz.
\newblock Sigma-delta ({$\Sigma$} {$\Delta$}) quantization and finite frames.
\newblock {\em IEEE Trans. Inform. Theory}, 52(5):1990--2005, 2006.

\bibitem{can99}
E.J. Cand{\`e}s.
\newblock Harmonic analysis of neural networks.
\newblock {\em Appl. Comput. Harmon. Anal.}, 6(2):197--218, 1999.

\bibitem{can05}
E.J. Cand{\`e}s and D.L. Donoho.
\newblock Continuous curvet transform: {II}. {D}iscretization and frames.
\newblock {\em Appl. Comput. Harmon. Anal.}, 19:198--222, 2005.

\bibitem{cas00}
P.G. Casazza.
\newblock The art of frame theory.
\newblock {\em Taiwanese J. Math.}, 4(2):129--201, 2000.

\bibitem{cas99}
P.G. Casazza, D.~Han, and D.R. Larson.
\newblock Frames for {B}anach spaces.
\newblock {\em Contemp. Math.}, 247:149--182, 1999.

\bibitem{cas04}
P.G. Casazza and G.~Kutyniok.
\newblock Frames of subspaces.
\newblock {\em Contemp. Math.}, 345:87--114, 2004.

\bibitem{chr13}
O.~Christensen.
\newblock {\em An Introduction to Frames and {R}iesz Bases}.
\newblock Birkh\"auser, Boston, 2003.

\bibitem{chr04}
O.~Christensen and Y.C. Eldar.
\newblock Oblique dual frames and shift-invariant spaces.
\newblock {\em Appl. Comput. Harmon. Anal.}, 17(1):48--68, 2004.

\bibitem{chr03}
O.~Christensen and D.~Stoeva.
\newblock p-frames in separable {B}anach spaces.
\newblock {\em Adv. Comput. Math.}, 18(2-4):117--126, 2003.

\bibitem{dau92}
I.~Daubechies.
\newblock {\em Ten lectures on wavelets}.
\newblock SIAM, Philadelphia, 1992.

\bibitem{dau86}
I.~Daubechies, A.~Grossmann, and Y.~Meyer.
\newblock Painless nonorthogonal expansions.
\newblock {\em J. Math. Phys.}, 27(5):1271--1283, 1986.

\bibitem{dor06}
M.~D{\"o}rfler, H.G. Feichtinger, and K.~Gr{\"o}chenig.
\newblock Time-frequency partitions for the {G}elfand triple $({S_0}, {L^2} ,
  {S'_0}) $.
\newblock {\em Math. Scand.}, 98(1):81--96, 2006.

\bibitem{dud98}
N.E. Dudey~Ward and J.R. Partington.
\newblock A construction of rational wavelets and frames in {H}ardy-{S}obolev
  space with applications to system modelling.
\newblock {\em SIAM J. Control Optim.}, 36(1):654--679, 1998.

\bibitem{duf52}
R.J. Duffin and A.C. Schaeffer.
\newblock A class of nonharmonic {F}ourier series.
\newblock {\em Trans. Amer. Math. Soc.}, 72(2):341--366, 1952.

\bibitem{eld02}
Y.C. Eldar and D.G. Forney.
\newblock Optimal tight frames and quantum measurement.
\newblock {\em IEEE Trans. Inform. Theory}, 48(3):599--610, 2002.

\bibitem{fei94}
H.G. Feichtinger and K.~Gr{\"o}chenig.
\newblock Theory and practice of irregular sampling.
\newblock Wavelets: mathematics and applications, 305--363, 1994.

\bibitem{for04}
M.~Fornasier.
\newblock Quasi-orthogonal decompositions of structured frames.
\newblock {\em J. Math. Anal. Appl.}, 289(1):180--199, 2004.

\bibitem{gab03}
J.-P. Gabardo and D.~Han.
\newblock Frames associated with measurable spaces.
\newblock {\em Adv. Comput. Math.}, 18(2-4):127--147, 2003.

\bibitem{gua06}
P.~G{\u{a}}vru{\c{t}}a.
\newblock On some identities and inequalities for frames in {H}ilbert spaces.
\newblock {\em J. Math. Anal. Appl.}, 321(1):469--478, 2006.

\bibitem{gro01}
K.~Gr\"ochenig.
\newblock {\em Foundations of Time-–Frequency Analysis}.
\newblock Birkh\"auser, Boston, 2001.

\bibitem{han00}
D.~Han and D.R. Larson.
\newblock {\em Frames, bases and group representations}.
\newblock Amer. Math. Soc. 697, 2000.

\bibitem{kaf09}
V.~Kaftal, D.~Larson, and S.~Zhang.
\newblock Operator-valued frames.
\newblock {\em Trans. Amer. Math. Soc.}, 361(12):6349--6385, 2009.

\bibitem{li08}
D.F. Li and W.C. Sun.
\newblock Some equalities and inequalities for generalized frames.
\newblock {\em Chinese J. Contemp. Math.}, 29(3):301--308, 2008.

\bibitem{li12}
J.Z. Li and Y.C. Zhu.
\newblock Some equalities and inequalities for g-{B}essel sequences in
  {H}ilbert spaces.
\newblock {\em Appl. Math. Lett.}, 25(11):1601--1607, 2012.

\bibitem{li04}
S.~Li and H.~Ogawa.
\newblock Pseudoframes for subspaces with applications.
\newblock {\em J. Fourier Anal. Appl.}, 10(4):409--431, 2004.

\bibitem{por16}
A. Poria.
\newblock Behavior of Gabor frame operators on Wiener amalgam spaces.
\newblock {\em Int. J. Wavelets Multiresolut. Inf. Process.}, 14(4):1650028, 15 pp., 2016.

\bibitem{por15}
A. Poria and J. Swain.
\newblock Hilbert space valued Gabor frames in weighted amalgam spaces.
\newblock preprint arXiv:1508.01646v2, 2015.

\bibitem{rin71}
J.R. Ringrose.
\newblock {\em Compact Non-Self-Adjoint Operators}.
\newblock Van Nostrand Reinhold Math. Studies 35, 1971.

\bibitem{sad12}
G.~Sadeghi and A.~Arefijamaal.
\newblock von {N}eumann--{S}chatten frames in separable {B}anach spaces.
\newblock {\em Mediterr. J. Math.}, 9(3):525--535, 2012.

\bibitem{sim79}
B.~Simon.
\newblock {\em Trace ideals and their applications}.
\newblock Cambridge University Press Cambridge, UK,, 1979.

\bibitem{str03}
T.~Strohmer and R.~W. Heath.
\newblock {G}rassmannian frames with applications to coding and communication.
\newblock {\em Appl. Comput. Harmon. Anal.}, 14(3):257--275, 2003.

\bibitem{sun06}
W.~Sun.
\newblock {G}-frames and g-{R}iesz bases.
\newblock {\em J. Math. Anal. Appl.}, 322(1):437--452, 2006.

\bibitem{xia08}
X.C. Xiao, Y.C. Zhu, and X.M. Zeng.
\newblock Some properties of g-{P}arseval frames in {H}ilbert spaces.
\newblock {\em Acta Math. Sinica (Chin. Ser.)}, 51(6):1143--1150, 2008.

\bibitem{yan09}
X.H. Yang and D.F. Li.
\newblock Some new equalities and inequalities for {G}-frames and their dual
  frames.
\newblock {\em Acta Math. Sinica (Chin. Ser.)}, 52(5):1033--1040, 2009.

\bibitem{you01}
R.M. Young.
\newblock {\em An Introduction to Non-Harmonic {F}ourier Series}.
\newblock Academic Press, New York, 1980.

\bibitem{zhu10}
X.~Zhu and G.~Wu.
\newblock A note on some equalities for frames in {H}ilbert spaces.
\newblock {\em Appl. Math. Lett.}, 23(7):788--790, 2010.

\end{thebibliography}

\end{document}